\documentclass[twoside,a4paper,reqno,12pt]{amsart}

\usepackage{cite}
\usepackage{mathrsfs}
\usepackage{amsmath}
\usepackage{latexsym}
\usepackage{amssymb,bm}
\usepackage{amsfonts}
\usepackage{longtable}
\usepackage{multirow}
\usepackage[colorlinks,linkcolor=cyan,anchorcolor=blue,citecolor=red,backref=page]{hyperref}
\usepackage{braket}
\usepackage{enumerate}
\usepackage{xcolor}
\usepackage{graphics}

\def\l{\langle}
\def\r{\rangle}
\def\ov{\overline}

\newcommand{\qd}[1]{{ \fbox{$ \boldsymbol{#1} $} }} 

\def\ZZ{\mathbb Z}
\def\la{\lambda}  \def\m{\mu}
\def\Aut{{\sf Aut}} \def\Cos{{\sf Cos}}
\def\A{{\rm A}} \def\S{\rm S}
\def\PSL{{\rm PSL}}
\def\SL{{\rm SL}}

\def\a{\alpha}
\def\b{\beta}

\def\C{{\bf C}}

\def\D{{\rm D}}

\newtheorem{thm}{Theorem}[section]

\newtheorem{construction}[thm]{Construction}

\newtheorem{theorem}[thm]{Theorem}
\newtheorem{lemma}[thm]{Lemma}

\newtheorem{corollary}[thm]{Corollary}
\newtheorem{problem}[thm]{Problem}
\newtheorem{example}[thm]{Example}
\theoremstyle{definition}
\newtheorem{rem}[thm]{Remark}

\def\qed{\nopagebreak\hfill{\rule{4pt}{7pt}}
	\medbreak}

\def\Ga{{\it\Gamma}}
\def\Sig{{\it\Sigma}}
\def\Ome{{\it\Omega}}

\def\Sp{{\rm Sp}}
\def\Sz{{\rm Sz}}

\def\S{{\mathrm{S}}}
\def\A{{\mathrm{A}}}

\def\Sym{{\mathrm{Sym}}}
\def\Alt{{\mathrm{Alt}}}

\def\Cos{{\mathrm{Cos}}}

\def\leq{\leqslant}
\def\le{\leqslant}
\def\geq{\geqslant}
\def\ge{\geqslant}

\begin{document}
	
	\title{Non-solvable $2$-arc-transitive covers of the Petersen graph}

	\author{Jiyong Chen}
	\author{Cai Heng Li}
	\author{Ci Xuan Wu}
	\author{Yan Zhou Zhu}
	
	\address{J.~Chen, School of Mathematical Sciences\\
		Xiamen University \\
		Xiamen 361005\\
		P. R. China}
	\email{chenjy1988@xmu.edu.cn}
	
	\address{C.H.~Li, SUSTech International Center for mathematics\\
		Department of Mathematics\\
		Southern University of Science and Technology \\
		Shenzhen 518055\\
		P. R. China}
	\email{lich@sustech.edu.cn}
	
	\address{C.X.~Wu, School of Statistics and Mathematics \\
		Yunnan University of Finance and Economics\\
		Kunming 650021\\
		P. R. China}
	\email{wucixuan@gmail.com}
	
	\address{Y.Z.~Zhu, School of Mathematical Sciences\\
		Xiamen University \\
		Xiamen 361005\\
		P. R. China}
	\email{zhuyz@xmu.edu.cn}
	
	\date\today

	\begin{abstract}
        While solvable covers of 2-arc-transitive graphs have been studied
extensively, non-solvable covers have remained largely unexplored. 
        This paper takes a first step by constructing such covers for the Petersen graph and determining their structural properties.
		
		\bigskip
		\noindent{\bf Keywords:} $2$-arc-transitive, non-solvable cover, minimal cover\\
		\noindent{\bf 2000 Mathematics subject classification:} 05C38, 20B25	
	\end{abstract}
	\maketitle

\section{Introduction}

Throughout this paper, all groups and graphs are finite, and graphs are also simple, undirected and regular.
Since valencies 1 and 2 yield only single edges and cycles, we assume valency at least 3.
Given a graph $\Ga$, an \emph{$s$-arc} of $\Ga$ is a sequence of $s+1$ vertices such that any two consecutive vertices are adjacent, and any three consecutive vertices are distinct. 
Let $\Aut(\Ga)$ denote the full automorphism group of $\Ga$.
If a subgroup $G \leqslant \Aut(\Ga)$ acts transitively on the set of $s$-arcs, then $\Ga$ is called \emph{$(G,s)$-arc-transitive}.
In this case, we also simply say that $\Ga$ is \emph{$s$-arc-transitive} or that the pair $(\Ga, G)$ is \emph{$s$-arc-transitive}.

A main topic in algebraic graph theory is the classification of $s$-arc-transitive pairs for $s\geqslant 2$.
Using local action analysis, Weiss~\cite{Weiss} proved the bound $s\leqslant 7$.
Praeger~\cite{Praeger1} introduced a global approach via quotient constructions to study $s$-arc-transitive graphs for $s\ge 2$.
As $(s+1)$-arc-transitivity implies $s$-arc-transitivity, it suffices to focus the classification on $2$-arc-transitive pairs, and we now describe Praeger's approach in this case.

Let $(\Ga,G)$ be a connected $2$-arc-transitive pair with vertex set $V$, and let $N \lhd G$.
The \emph{normal quotient} $\Ga_N$ is the graph whose vertices are the $N$-orbits,  where two orbits are adjacent if there exists an edge in $\Ga$ joining a vertex in one to a vertex in the other.
A key result of Praeger~\cite{Praeger1} states that if $N$ has at least three orbits on $V$, then $\Ga_N$ is $(G/N,2)$-arc-transitive and has the same valency as $\Ga$.
In this situation, $(\Ga, G)$ is called a $2$-arc-transitive \emph{normal cover} (or simply a \emph{cover}) of $(\Ga_N, G/N)$, and $N$ is called the \emph{transformation group} of the cover; we also say $\Ga$ is a cover of $\Ga_N$.
For the general theory of quotients and covers of graphs, we refer to~\cite{gross1987Topological}.

The quotient construction gives a natural reduction: any connected $2$-arc-transitive pair can be reduced to a pair for which no non-trivial normal subgroup has at least three orbits. 
For a non-bipartite graph $\Ga$, a pair $(\Ga, G)$ has this property if and only if $G$ is \emph{quasiprimitive} on the vertex set (meaning every non-trivial normal subgroup is transitive), and we say that the pair is \emph{quasiprimitive} in this case.
Thus every connected non-bipartite $2$-arc-transitive pair is a normal cover of some quasiprimitive $2$-arc-transitive pair.
Based on this observation, Praeger~\cite{Praeger1} suggested her program to study the class of connected non-bipartite $2$-arc-transitive graphs in two steps:
\begin{enumerate}
    \item characterize connected quasiprimitive $2$-arc-transitive pairs;
    \item characterize connected normal covers of quasiprimitive $2$-arc-transitive pairs.
\end{enumerate}

For step (1), Praeger~\cite{Praeger1} gave a systematic description of connected quasiprimitive $2$-arc-transitive pairs, showing that the quasiprimitive type of $G$ is one of the four possibilities among the eight quasiprimitive types defined in~\cite{praeger97-8types}.
Substantial progress has been made for each of these types; see~\cite{IP, Baddeley, FP1, FP2, HNP, PA-1, LLW} for instance.

For step~(2), a classical result of Conway~\cite[Chapter~19]{Biggs-book} already shows that every connected $s$-arc-transitive graph ($s\geq 1$) admits a connected $s$-arc-transitive cover with an elementary abelian transformation group.
Consequently, as far as we know, all work towards step~(2) studies covers with solvable transformation groups, such as abelian groups of small rank and metacyclic groups; see~\cite{du20052arctransitive, DXY, CM1, CM2, Ele_abel_cover, penta_graph, XDKX, metacyc_cover} and the references therein.
However, very little is known about covers with non-solvable transformation groups, and we call such covers \emph{non-solvable covers}.
This naturally raises the following problem.

\begin{problem}
Given a connected $(G,2)$-arc-transitive graph $\Sig$, does there exist a connected non-solvable cover $(\Ga, N.G)$ of $(\Sig, G)$?
\end{problem}

This paper establishes the existence of connected $2$-arc-transitive non-solvable covers of the Petersen graph; also see our subsequent paper~\cite{complete_cover} for such covers of complete graphs.
Here, a group $P$ is called \emph{$(2,3)$-generated} if $P = \langle x, y \rangle$ with $|x| = 2$ and $|y| = 3$, and \emph{perfect} if $P = P'$, where $P'$ is the commutator subgroup of $P$.
\begin{theorem}\label{Main Theorem}
	Let $P$ be a $(2,3)$-generated perfect group, $\Sig$ be the Petersen graph, and $X= P\wr\A_5=P^5{:}\A_5$. 
	Then $(\Sig,\A_5)$ has a connected $2$-arc-transitive cover $(\Ga,X)$ with transformation group $P^5$.
	Moreover, the following properties hold:
	\begin{itemize}
		\item[(i)]$\Aut(\Ga)=X$;
		\item[(ii)] $\Ga$ is a non-Cayley graph;
		\item[(iii)] $\Ga$ has girth $10$.
	\end{itemize}
\end{theorem}
\begin{rem}
	In Construction~\ref{cons-1}, we outline the detailed construction of the graph $\Ga$ presented in Theorem~\ref{Main Theorem}.   
	The construction takes as input a $(2,3)$-generated perfect group $P$.
	There are many such groups, for example:

	\begin{enumerate}[(1)]
		\item Non-abelian simple groups are perfect.
		By Liebeck and Shalev~\cite{23-GEN}, most of them are $(2,3)$-generated, more precisely, except for $\Sz(q)$, $\Sp(4,q)$ with $q$ even, and finitely many other groups, every  non-abelian simple group is $(2,3)$-generated.
		
		\item  If $T_1,\dots, T_d$ are pairwise non-isomorphic non-abelian simple $(2,3)$-generated groups, then the direct product $T_1\times \dots \times T_d$ is also $(2,3)$-generated, as demonstrated in Example~\ref{DirectProductOfSimple}.
		
		\item Construction~\ref{cons-1} can be iterated, namely, if $P$ is perfect and $(2,3)$-generated, then by~Lemma~\ref{cubic-connected}, so are $P\wr\A_5$, $(P\wr\A_5)\wr\A_5$, and so on.

		\item 
		If $T$ is a perfect $(2,3)$-generated group, and $N.T = \mathbb{Z}_p^\ell.T$ is a non-split extension with $p>3$ and $N$ minimal normal in $N.T$, then $N.T$ is also perfect and $(2,3)$-generated; see Lemma~\ref{lem:nonsplit-23}.
		For $T = \PSL(n,p)$, an explicit construction of a non-split extension $P = \ZZ_p^{n^2-1}{.}\PSL(n,p)$ can be found in~\cite{CLWZ:Irr-covers}.
        Thus we have the following corollary.
	\end{enumerate}
\end{rem}

\begin{corollary}\label{coro-nonsplit}
	Let $P= \ZZ_p^{n^2-1}{.}\PSL(n,p)$ be a non-split extension, where $n\geqslant 2$, $p>3$, $p\nmid n$ and $\PSL(n,p)$ is $(2,3)$-generated.
	Set $X=P\wr \A_5$.
	Then the Petersen graph has a connected $(X,2)$-arc-transitive cover with transformation group $P^5$.
\end{corollary}

The above remarks show that the group $P$ can be chosen from a very wide range of perfect groups. 
A complete classification of non-solvable covers is therefore a ``wild'' problem. 
Note that any normal cover can be broken into a series of normal covers whose transformation groups are minimal normal subgroups. 
This leads naturally to the study of minimal covers.
Formally, given a normal cover $(\Ga, X)$ of $(\Sig, G)$ with transformation group $N$, we call it a \emph{ minimal cover} if $N$ is a minimal normal subgroup of $X$.
It is well-known that such an $N$ is a direct product of pairwise isomorphic simple groups. 
If $N$ is non-abelian, then $N = T^k$ with $T$ non-abelian simple, and $G$ acts transitively on the $k$ direct factors by conjugation. 
Hence $k$ is a transitive permutation degree of $G$.
Our main result shows that for any $(2,3)$-generated non-abelian simple group $T$, the Petersen graph admits a minimal cover with transformation group $T^5$ (the case $k=1$ is also possible; see Example~\ref{ex:diag}). 
This motivates the following problem.
\begin{problem}
	Given a non-abelian simple group $T$, for which transitive permutation degrees $k$ of $\A_5$ does there exist a minimal cover of the Petersen graph with transformation group $T^k$?
\end{problem}

\section{Preliminaries on groups and coset graphs}

We first collect some facts about $(2,3)$-generated perfect groups, which will be needed for the construction of $2$-arc-transitive covers of the Petersen graph.
As mentioned before, most finite simple groups are $(2,3)$-generated by a result of Liebeck and Shalev~\cite{23-GEN}.
The following example gives a pair of $(2,3)$-generators for $\PSL(2,p)$.

\begin{example}\label{const-PSL}
	{\rm
		Let $p\geq 5$ be a prime. 
		Define elements $\tilde{a}, \tilde{b} \in \SL(2,p)$ by
		\[
		\tilde{a} = \begin{pmatrix} 0 & -1 \\ 1 & -1 \end{pmatrix}\mbox{ and }
		\tilde{b} = \begin{pmatrix} 0 & 1 \\ -1 & 0 \end{pmatrix}.
		\]
		Computation shows that $\tilde{a}^3 = \tilde{b}^4= 1$. 
		Note that the products $\tilde{a}\tilde{b} = \begin{pmatrix} 1 & 0 \\ 1 & 1 \end{pmatrix}$ and $\tilde{b}\tilde{a} = \begin{pmatrix} 1 & -1 \\ 0 & 1 \end{pmatrix}$ both have order $p$. 
        Now, set $T=\SL(2,p)/{\bf Z}(\SL(2,p))= \PSL(2,p)$ and let $a, b$ be the images of $\tilde{a},\tilde{b}$ in $T$. It follows that $$a^3=b^2=(ab)^p=(ba)^p=1.$$
        Since no proper subgroup of $\PSL(2, p)$ contains two distinct subgroups of order $p$, it follows that $T=\l a,b\r $. 
		\qed }
\end{example}

The next example shows that the direct product of pairwise non-isomorphic $(2,3)$-generated non-abelian simple groups is $(2,3)$-generated.

\begin{example}\label{DirectProductOfSimple}
	{\rm
		Let $d\ge2$ be an integer, and let $T_1, \dots, T_d$ be pairwise non-isomorphic $(2,3)$-generated non-abelian simple groups. More specifically, assume
		\[\mbox{$T_i=\l\mu_i,\lambda_i\r$, where $|\mu_i|=2$ and $|\lambda_i|=3$.}\]
		Now, let $P=T_1\times \dots \times T_d$ and define $\lambda=(\lambda_1,\dots, \lambda_d)$ and $\mu=(\mu_1,\dots,\mu_d)$. It follows immediately that $|\lambda|=3$ and $|\mu|=2$. Moreover, as the projection from $\l \lambda, \mu\r$ to each of the direct factors is surjective and $T_i$ are pairwise non-isomorphic, we have $P=\l \lambda, \mu\r$ which is a $(2,3)$-generated perfect group.  
		\qed
	}
\end{example}

Non-split extensions provide another source of $(2,3)$-generated groups.

\begin{lemma}\label{lem:nonsplit-23}
Let $T$ be a $(2,3)$-generated group, and let $N.T = \mathbb{Z}_p^\ell.T$ be a non-split extension with $p>3$ and $N$ minimal normal in $N.T$. Then $N.T$ is $(2,3)$-generated. Moreover, if $T$ is perfect, then $N.T$ is also perfect.
\end{lemma}
\begin{proof}
Since $T$ is $(2,3)$-generated, there exist elements $a,b \in T$ such that $|a|=2$, $|b|=3$, and $\langle a,b\rangle = T$. Let $\pi: N.T \to T$ be the natural epimorphism, and choose preimages $\tilde a, \tilde b \in N.T$ such that $\pi(\tilde a)=a$ and $\pi(\tilde b)=b$. Then $\tilde a$ has order either $2$ or $2p$, and $\tilde b$ has order either $3$ or $3p$. Since $p>3$, there exist elements $x \in \langle \tilde a \rangle$ and $y \in \langle \tilde b \rangle$ such that $|x|=2$, $|y|=3$, $\pi(x)=a$, and $\pi(y)=b$. 
It follows that $N\langle x,y\rangle = N.T$, as $\langle x,y\rangle$ maps onto $T$ under $\pi$. 

We claim that $N.T = \langle x,y\rangle$. Assume the contrary. Then $\langle x,y\rangle < N.T$, which implies that $N \cap \langle x,y\rangle$ is a proper subgroup of $N$ normalized by $\langle x,y\rangle$. By the minimality of $N$, $N \cap \langle x,y\rangle = 1$, implying that $N.T$ splits over $N$. This contradicts the assumption that $N.T$ is a non-split extension. Therefore, $\langle x,y\rangle = N.T$, and $N.T$ is $(2,3)$-generated.

Assume further that $T$ is perfect, and let $Y=(N.T)'$. Since $\pi(Y)=T$, we obtain $N Y = N.T$. If $Y < N.T$, then $Y \cap N$ is a proper normal subgroup of $N.T$. The minimality of $N$ implies $Y \cap N = 1$, which again yields a split extension over $N$, a contradiction. Thus, $N \le Y$, yielding $Y = N.T$. This proves that $N.T$ is perfect.
\end{proof}

We now introduce the coset graph construction, a standard tool in the study of arc-transitive graphs.
Let $G$ be a group, $H$ a subgroup of $G$, and $g\in G$ with $g^2\in H$. Define the \emph{coset graph} $\Ga = \Cos(G,H,HgH)$ to be the graph with vertex set $[G:H] = \{ Hx \mid x\in G \}$, where two vertices $Hx$ and $Hy$ are adjacent if and only if $yx^{-1}\in HgH$.  
The group $G$ acts naturally on $\Ga$ by right multiplication, and this action is arc-transitive. 
Under this action, the stabilizer of the vertex $H$ is exactly $H$, and the element $g$ interchanges the two adjacent vertices $H$ and $Hg$.
Furthermore, this action is faithful (meaning $G \le \Aut(\Ga)$) if and only if $H$ contains no non-trivial normal subgroup of $G$.

Conversely, every arc-transitive graph arises in this way. 
Indeed, let $\Ga$ be a $G$-arc-transitive graph with vertex set $V$.
Fix an arc $(\alpha,\beta)$, and choose $g\in G$ interchanging $\alpha$ and $\beta$.
Then $g^2\in G_\alpha$. 
Since $G$ acts transitively on $V$, we may identify $V$ with  the set $[G:G_\alpha]$ of right cosets via the action by right multiplication: for $y\in G$,
\[
y:\ G_\alpha x \mapsto G_\alpha xy.
\]
Under this identification, $G_\alpha$ corresponds to $\alpha$, and $G_\alpha g$ to $\beta$. The neighbourhood of $\alpha$ is then
\[
\Ga(\alpha)=\{G_\alpha gh \mid h\in G_\alpha\},
\]
and hence adjacency is determined by
\[
G_\alpha x \sim G_\alpha y \iff yx^{-1}\in G_\alpha g G_\alpha.
\]
Thus $\Ga$ is isomorphic to the coset graph $\Cos(G,G_\alpha,G_\alpha g G_\alpha)$, as defined above. 

The following lemma is folklore.

\begin{lemma}\label{CosetGraph}
	Let $\Ga=\Cos(G, H, HgH)$ be a coset graph.
	Then $\Ga$ is $G$-arc-transitive and the following statements hold.
	\begin{enumerate}[{\rm(i)}]
		\item $\Ga$ is connected if and only if $\braket{H,g}=G$;
		\item $\Ga$ is $(G,2)$-arc-transitive if and only if $H$ is $2$-transitive on $[H:H\cap H^g]$.
	\end{enumerate}
\end{lemma}

Now, we write the Petersen graph as a coset graph.

\begin{example}\label{ex:Petersen}
	{\rm
		Let $\Ome=\{1,2,3,4,5\}$, and let $\Ome^{\{2\}}$ be the set of the ten 2-subsets of $\Ome$.
        Let $\Sig$ be the Petersen graph, defined here as the graph with vertex set $\Ome^{\{2\}}$ where two vertices $\a, \b \in \Ome^{\{2\}}$ are adjacent if and only if $\a \cap \b = \emptyset$. 
		Fix adjacent vertices $\a=\{4,5\}$ and $\b=\{1,2\}$.
		
		If we take $G:=\Aut\Sig=\Sym(\Ome)\cong\S_5$, then $\Sig$ is a $(G,3)$-arc-transitive graph, with stabilizers $G_\a=\Sym\{1,2,3\}\times\Sym\{4,5\}\cong\S_3\times\S_2$, and $G_{\a\b}=\l(12),(45)\r\cong \ZZ_2^2$ such that $g=(14)(25)$ interchanges $\a$ and $\b$.
		Thus the Petersen graph $\Sig$ can also be written as the coset graph $\Cos(G,G_\a,G_\a gG_\a)$.
		
		On the other hand, if we take $G':=\Alt(\Ome)\cong \A_5\leq \Aut\Sig$, then $\Sig$ is a $(G',2)$-arc-transitive graph,
		with $G'_\a=\l (123), (12)(45) \r\cong \S_3$, and $G'_{\a\b}=\l(12)(45)\r\cong \ZZ_2$ such that $g=(14)(25)$ interchanges $\a$ and $\b$.
		Thus the Petersen graph $\Sig$ can be also written as the coset graph $\Cos(G',G'_\a,G'_\a gG'_\a)$.
		\qed
	}
\end{example}

Finally, we give a $2$-arc-transitive non-solvable cover of the Petersen graph constructed from the direct product $\PSL(2,p)\times \A_5$.

\begin{example}\label{ex:diag}
	{\rm
		Let $T= \PSL(2,p)$, where $p>5$ is a prime with $p\equiv 1\pmod{4}$.
		Let $a,b,c\in T$ be the projective images of the matrices
		\[\begin{pmatrix}0&-1\\ 1&-1\end{pmatrix},\quad \begin{pmatrix}0&1\\-1&0\end{pmatrix}\quad\mbox{and}\quad\begin{pmatrix}0&\lambda\\ \lambda&0\end{pmatrix}, \]
		respectively, where $\lambda\in\mathbb{F}_p^{\times}$ has order $4$.
		Then $|a|=3$, $|b|=|c|=2$, $a^c=a^{-1}$ and $bc=cb$.
		Set $X=T\times \A_5$.
        Let $H<X$ and $\gamma\in X$ be given by 
		\[H=\l\bigl(a,(123)\bigr),\bigl(c,(12)(45)\bigr)	\r\cong\S_3\mbox{ and }\gamma=\bigl(b,(14)(25)\bigr).\]
		Define $\Ga=\Cos(X,H,H\gamma H)$.

        Example~\ref{const-PSL} shows that $T=\l a,b\r$.
		Then $\l H,\gamma\r$ projects onto each direct factor of $X=T\times \A_5$.
		Since $T=\PSL(2,p)\not\cong \A_5$, we have that $\langle H,\gamma\rangle = X$, so $\Ga$ is connected.
		Note that $\Ga_T$ is isomorphic to the Petersen graph as constructed in Example~\ref{ex:Petersen}.
		Then $T$ has $10$ orbits on vertices of $\Ga$.
		Moreover, $H^\gamma\cap H = \langle (c,(12)(45)) \rangle \cong \mathbb{Z}_2$, so $\Ga$ has valency $3$ and is $2$-arc-transitive.
		Thus Praeger's theorem shows that $\Ga$ is a cover of $\Ga_T$.
		Therefore, $\Ga$ is a connected $2$-arc-transitive cubic graph, and is a non-solvable cover of the Petersen graph.
		\qed
	}
\end{example}

\section{Proof of Theorem~\ref{Main Theorem}}\label{sec:petersencover}

To prove Theorem~\ref{Main Theorem}, we construct explicit $2$-arc-transitive non-solvable normal covers of the Petersen graph. 
Recalling Example~\ref{ex:Petersen}, we denote the Petersen graph by $\Sig$, which is $(\A_5,2)$-arc-transitive.

\begin{construction}\label{cons-1}
	{\rm
		Let $P=\braket{\mu,\la}$ be a perfect group, where $|\mu|=2$ and $|\la|=3$.
		Let $X=P\wr \A_5=P^5:\A_5$ be such that, for $(t_1,t_2,t_3,t_4,t_5)\in P^5$ and $\pi\in\A_5$,
		\[(t_1,t_2,t_3,t_4,t_5)^{\pi^{-1}}=(t_{1^\pi}, t_{2^\pi}, t_{3^\pi}, t_{4^\pi}, t_{5^\pi}).\]
		Pick $h_1,h_2,g\in X$ as follows:
		\[\begin{array}{l}
			h_1=(1,1,1,\la,\la^{-1})(123),\\
			h_2=(\m,\m,\m,1,1)(12)(45),~\mbox{and}~ \\
			\ \, g=(1,\m,1,1,\m)(14)(25).
		\end{array}\]
		Let $H=\braket{h_1,h_2}$ and  $\Ga=\Cos(X,H,HgH)$.
		\qed
	}
\end{construction}

The graph $\Ga$ defined above is a connected $2$-arc-transitive non-solvable cover of the Petersen graph with $\Aut(\Ga)=X$, and it is non-Cayley with girth $10$, as stated in Theorem~\ref{Main Theorem}. 
These properties will be verified in Lemmas ~\ref{lem:covers}, \ref{cubic-connected}, \ref{Aut(Gamma)}, \ref{Non-Cayley} and~\ref{girth}.

Throughout the rest of this section, we use the notation in Construction~\ref{cons-1} and begin with some basic properties.

\begin{lemma}\label{lem:covers}
    The following statements hold:
	\begin{enumerate}[{\rm(1)}]
		\item $H=\langle h_1\rangle{:}\l h_2\r\cong\S_3$, and $|g|=2$;
		\item $H\cap H^g=\l h_2\r\cong \S_2$;
		\item $\Ga=\Cos(X,H,HgH)$ is a $2$-arc-transitive cubic graph;
		\item $(\Ga, X)$ is a cover of $(\Sig,\A_5)$ with transformation group $P^5$.
	\end{enumerate}
\end{lemma}
\begin{proof}
(1). It is easily shown that $h_1$ is of order 3, and $h_2,g$ are involutions.
Further, $h_1^{h_2}=h_1^{-1}$, and so $H=\l h_1,h_2\r\cong \S_3$.

(2). Since $h_2g=gh_2$, we obtain that $H\cap H^g=\l h_2\r\cong \S_2$.

(3). Obviously, $H\cong \S_3$ acts 2-transitively on $[H:H\cap H^g]$.
Thus $X$ acts 2-arc-transitively on the graph $\Ga$.

(4). Note that $\overline X=X/P^5\cong \A_5$.
Let $\overline H$ and $\overline g$ be the homomorphic images of $H$ and $g$, respectively.
Then $\overline H=\l(123),(12)(45)\r\cong \S_3$ and $\ov g=(14)(25)$. By Example~\ref{ex:Petersen}, the normal quotient $\Ga_{P^5}=\Cos(\ov X,\ov H,\ov H\ov g\ov H)$ is isomorphic to the Petersen graph.
\end{proof}

\begin{lemma}\label{ generated by commutators}
	The perfect group $P=\l \m,\lambda\r$ is generated by the commutators $[\m,\la]$ and $[\m,\la^2]$.
\end{lemma}
\begin{proof}
Let $R=\braket{[\m,\la],[\m,\la^2]}$.
Since $|\mu|=2$,  conjugation by $\mu$ inverts both $[\m,\la]$ and $[\m,\la^2]$, implying that $R^\mu \leq R$.
Since $|\la|=3$,  it follows that $[\m,\la^2]^{\la}=[\la,\mu]=[\mu,\la]^{-1}$.
Note that $[\mu,\la]^\la=[\mu,\la]^{-1}[\mu,\la^2]$.  Consequently, $R^\la \leq R$, and thus $R\unlhd P$.
Since $P/R=\braket{\mu R,\la R}$ is abelian and $P$ is perfect, we conclude that $R=P$.
\end{proof}

To facilitate the proofs of the subsequent lemmas, we introduce the following notation.
Let 
\[ N = P^5 = \{(p_1, p_2, p_3, p_4, p_5) \mid p_i \in P \text{ for } 1 \leq i \leq 5\}, \]
and for $1 \leq i \leq 5$, define
\[ P_i = \{(p_1, p_2, p_3, p_4, p_5) \mid p_i \in P \text{ and } p_j = 1 \text{ for } j \neq i\}. \]
Set $G=\braket{h_1,h_2,g}$ and $K=G\cap N$. 
The following lemma establishes key properties of the subgroups $G$ and $K$.

\begin{lemma}\label{property G}
The following statements hold.
	\begin{itemize}
		\item [(1)] 
		The action of $G$ on the set $\{P_1,\dots,P_5\}$ by conjugation is permutation isomorphic to the natural action of $\A_5$ with kernel $K$. In particular, we have $G/K\cong \A_5$. 
		\item[(2)] If $S$ is a subset of $\{1,\dots,5\}$ with $|S|\in \{1,2\}$, then $K$ projects surjectively onto $\prod_{i\in S}P_i$.
	\end{itemize} 
\end{lemma}
\begin{proof}
	(1).  
	Note that the conjugacy action of $N$ on the set $\{P_1,\dots,P_5 \}$ is trivial. Hence, the action of $G$ is permutation isomorphic to the natural action of $\A_5\cong X/N=GN/N\cong G/K$ with kernel $K$, which gives part (1).
	
	(2). We begin by presenting some computational results.
	Let $x=(h_1g)^5$, $x_1=(\m,1,1,\la,\la^{-1}\m)$ and $x_2=(15234)$. Then
	$x=(x_1x_2)^5$ as $h_1g=x_1x_2$. Therefore,
	\begin{equation*}
		\begin{split}
			x&=(x_1x_2)^5=x_1\cdot (x_1)^{x_2^4}\cdot (x_1)^{x_2^3}\cdot (x_1)^{x_2^2}\cdot (x_1)^{x_2}\\
			&=(\m,1,1,\la,\la^{-1}\m)(\la^{-1}\m,1,\la,\m,1)(1,\la,\m,\la^{-1}\m,1)
			(1,\m,\la^{-1}\m,1,\la)(\la,\la^{-1}\m,1,1,\m)\\
			&=(\m\la^{-1}\m\la,\la\m\la^{-1}\m,\la\m\la^{-1}\m,\la\m\la^{-1}\m,\la^{-1}\m\la\m)\\
			&=([\m,\la],[\la^2,\m],[\la^2,\m],[\la^2,\m],[\la,\m]).\\
		\end{split}
	\end{equation*}
	Let $y=(xh_2)^2$. Since $\m$ is an involution, we have $\m[\m,\la^i]\m=[\la^i,\m]$ for $i\ge 0$. Therefore,
	\begin{align}
		y&=(([\m,\la]\m,[\la^2,\m]\m,[\la^2,\m]\m,[\la^2,\m],[\la,\m])(12)(45))^2 \notag\\
		&=([\m,\la]\m,[\la^2,\m]\m,[\la^2,\m]\m,[\la^2,\m],[\la,\m])([\la^2,\m]\m,[\m,\la]\m,[\la^2,\m]\m,[\la,\m],[\la^2,\m])\notag\\
		&=([\m,\la][\la^2,\m]^{\m},[\la^2,\m][\m,\la]^{\m},[\la^2,\m][\la^2,\m]^\m,[\la^2,\m][\la,\m],[\la,\m][\la^2,\m]) \notag\\
		&=([\m,\la][\m,\la^2],[\la^2,\m][\la,\m],1,[\la^2,\m][\la,\m],[\la,\m][\la^2,\m]). \label{eqn-y}
	\end{align}
	Moreover, $x^{-1}y=([\m,\la^2],[\la,\m],[\m,\la^2],[\la,\m],[\la^2,\m])$.
	
	To prove part (2), by the action of $G$ on $\{P_1,\dots,P_5\}$, it suffices to show that the projections from $K$ to $P_1$ and $P_1\times P_3$ are surjective.
	By Lemma~\ref{ generated by commutators}, the subgroup $\braket{x,y}$ projects surjectively onto $P_1$, and so does $K$.
	Therefore, $K$ projects surjectively onto each $P_i$ for $1\le i\le 5$.
	
	Let $K_0$ be the image of $K$ under the projection  onto the direct product factor  $P_1 \times P_3$ of $N$, and let $\pi_1$ and $\pi_3$ denote the projection maps from $K_0$ onto $P_1$ and $P_3$, respectively. Then $\pi_1(K_0) = P_1 \cong P$ and $\pi_3(K_0) = P_3 \cong P$.
	Since $K_0\cap P_1=\ker \pi_3\lhd K_0$, we have $K_0\cap P_1=\pi_1(K_0\cap P_1) \lhd \pi_1(K_0)=P_1$.
	As shown in Equation~\ref{eqn-y}, the image of $y$ under the projection to $P_1 \times P_3$ is  $ ([\mu, \lambda][\mu, \lambda^2], 1) $,  which shows that  $[\mu, \lambda][\mu, \lambda^2] \in  K_0 \cap P_1 $. By Lemma~\ref{ generated by commutators}, we have $ P = \langle [\mu, \lambda], [\mu, \lambda][\mu, \lambda^2] \rangle $, which implies that $ P_1/(K_0 \cap P_1) $ is cyclic.
	Since $P$ is perfect, it follows that $ K_0 \cap P_1 = P_1$. Consequently, $ K_0 = P_1 \times P_3$, completing the proof of part~(2).
\end{proof}

For any normal subgroup $Q \lhd P$ and $1 \leq i \leq 5$, define the subgroup $Q_i \lhd P_i$ as
\[Q_i=\{(p_1, p_2, p_3, p_4, p_5) \mid p_i \in Q \text{ and } p_j = 1 \text{ for } j \neq i\}.\]
Under this notation, $X$ acts transitively on $\{Q_1, \dots, Q_5\}$ by conjugation. 
Let $Q^5 = Q_1 \times \cdots \times Q_5$.

\begin{lemma}\label{cubic-connected}
    The graph $\Ga$ is connected,  $X=\braket{g,h_1}$, and $X$ is also a $(2,3)$-generated perfect group. 
\end{lemma}
\begin{proof}
	By Lemma~\ref{CosetGraph}, $\Ga$ is connected if and only if $X=G$.
	Let $P$ be a minimal counterexample to the statement $X= G$.  We will derive a contradiction through the following four steps.

	{\bf \noindent Step 1: $Q^5K=N$ holds for every $1\neq Q\lhd P$.} 

	If $Q=P$, then $X=NG=Q^5G$. 
	If $Q<P$, then the quotient group $P/Q$ is perfect and generated by $\mu Q$ and $\la Q$.
	Then $X/Q^5\cong  (P/Q)\wr \A_5$, and by assumption, $X/Q^5=GQ^5/Q^5$. Therefore, in either case, $X=GQ^5$, and it follows that $$N=GQ^5\cap N=(G\cap N)Q^5=KQ^5.$$

	{\bf \noindent Step 2: The center of $P$  is trivial.}
	
	 Let $Z$ be the center of $P$. If $Z\neq 1$, from Step~1, we have $N=Z^5K$. Thus, $N=N'=(Z^5K)'=K'\leq K\leq N$, and hence $N=K$. Recalling $G/K\cong X/N\cong \A_5$,
     we have $X=G$. This contradicts the assumption that  $P$ is a counterexample. Hence, $Z$ is trivial.

	{\bf \noindent Step 3: The projection   from $K$ to $P_1\times P_2\times P_3$ is surjective.}

	Assume the projection   from $K$ to $P_1\times P_2\times P_3$ is not surjective.
	Let $U=P_1\times P_2\times P_3$ and $V=P_4\times P_5$ and  let $K_0$  be the image under the projection from $K$ to $U$. Then $K_0\cong KV/V$ and $K_0V=KV<N$.
	 Let $\pi_1$, $\pi_3$, $\pi_{12}$ and $\pi_{23}$  denote the projection maps from $K_0$ onto $P_1$, $P_3$,  $P_1\times P_2$ and $P_2\times P_3$, respectively.
	According to Lemma~\ref{property G}(2), $\pi_{12}(K_0) =P_1\times P_2$ and $\pi_3(K_0) =P_3$.
	Since $|P_1\times P_2|> |P_3|$, 
	\[|K_0 \cap (P_1\times P_2)| =|\ker \pi_3|>|\ker \pi_{12}| \geq 1.  \]
	This implies that $K_0\cap (P_1\times P_2)\neq 1$.

	Suppose $K_0 \cap P_1 \neq 1$. 
	Note that $K_0\cap P_1$ is the kernel of $\pi_{23}$. It follows that $K_0 \cap P_1=\pi_1(K_0 \cap P_1)\lhd \pi_1(K_0)=P_1$. 
	Let $M$ be the non-trivial normal subgroup of $P$ such that the corresponding subgroup $M_1$ equals $K_0\cap P_1$.
	Thus,  $M_1 \leq K_0 \leq K_0V=KV$.
	By Lemma~\ref{property G}, the conjugacy action of $G$ on the set $\{P_1, \dots, P_5\}$ is permutation isomorphic to $\A_5$.
	Let $G_{\{4,5\}}$ be the set stabilizer of $G$ on $\{P_4, P_5\}$. Then $G_{\{4,5\}}$ normalizes both $K$ and $V$ and acts transitively on $\{P_1, P_2, P_3\}$ and hence on $\{M_1,M_2,M_3\}$. We deduce that $M_1 \times M_2 \times M_3 \leq KV$.  Then $M^5\leq KV$ as $M_4 \times M_5 \leq V$, and hence $M^5K\leq KV$. 
	By Step~1, we have $N =M^5K$, which contradicts $KV<N$.
	
	Therefore, we must have $K_0 \cap P_1 = 1$, and similarly, $K_0 \cap P_2 = 1$. This gives $$(K_0 \cap (P_1 \times P_2)) \cap P_1 = K_0 \cap P_1 = 1.$$ Moreover, as $K_0\cap (P_1\times P_2) =\ker \pi_3 \lhd K_0$, 
	\[ K_0\cap (P_1\times P_2) =\pi_{12}(K_0\cap (P_1\times P_2))\lhd \pi_{12}(K_0)=P_1\times P_2. \]
	It follows that $K_0 \cap (P_1 \times P_2)$ centralizes $P_1$.  Recalling that $Z(P_1)=1$, the centralizer of $P_1$ in $P_1 \times P_2$ is $ P_2$.  Thus $K_0 \cap (P_1 \times P_2) \leq P_2$.  By symmetry, $K_0 \cap (P_1 \times P_2) \leq P_1$.
	Consequently, $K_0 \cap (P_1 \times P_2) \leq P_1 \cap P_2 = 1$, leading to a contradiction.
	Therefore, the projection from $K$ to $P_1 \times P_2 \times P_3$ is surjective.

	{\bf \noindent Step 4: $X=G$.}
	
	Now, $N=U\times V$ and $K$ projects surjectively onto both $U$ and $V$. 
	This gives $UK=N=VK$. Then  $K\cap U\neq 1$ as $|U|>|V|$. 
	Note that $K\cap P_1=\pi_1(K\cap P_1)\lhd \pi_1(K)=P_1$ as $K\cap P_1\lhd K$. 
	Similarly, $K\cap U=\pi(K\cap U)\lhd \pi(K)=U$ as $K\cap U\lhd K$, where $\pi$ is the projection from $U\times V$ to $U$.
	If $K\cap P_1\neq 1$, then $L:=\prod_{i=1}^{5}(K\cap P_i)$ is normal in $X$.
	By our assumption, $X/L=G/L$, and hence $X=G$.
	Now assume that $K\cap P_i=1$ for each $1\leq i\leq 5$. 
	Thus $(K \cap U)\cap P_1=K\cap P_1=1$. It follows that $K\cap U$ centralizes $P_1$, and so $K\cap U\leq P_2\times P_3$.
	Using similar arguments for $K\cap P_2$ and $K\cap P_3$, we obtain that $K\cap U$ is a subgroup of both $ P_1\times P_3$  and $P_1\times P_2$.
	Consequently, $K\cap U=1$, which is a contradiction.
	Thus $N=K$ and $X=G$, as we required.
	
	By the calculation in the proof of Lemma~\ref{lem:covers}, $h_2$ normalizes $\langle g,h_1\rangle$.
	Then $X/\langle g,h_1\rangle\leq \ZZ_2$. 
	As $X=P\wr \A_5$ is perfect, we conclude that $X=\langle g,h_1\rangle$.
	Therefore, $X$ is a $(2,3)$-generated group.
\end{proof}

\begin{lemma}\label{N:largestenormal}
		The subgroup $N$ is the unique maximal normal subgroup of $X$.
	In particular, $N$ is a characteristic subgroup of $X$.
\end{lemma}
\begin{proof}
	Let $M$ be a maximal normal subgroup of $X$. 
	Suppose that $M\neq N$.
	Then, by the maximality of $M$, $ X = MN $.
	Therefore, there exist $m\in M$ and $n\in N$ such that $mn=(1 2 3 4 5)$.
	If $P_1\leq M$, then $P_{i+1}=P_1^{m^i}\leq M^{m^i}=M$. 
	This gives $N=P^5\leq M$ and $X=MN=M$, a contradiction.
	If $P_1\not\leq M$, then let $Q_1=P_1\cap M$. For any normal subgroup $R_1$ of $P_1$ satisfying $Q_1< R_1\leq P_1$, we have $M/M \neq R_1M/M  \lhd NM/M=X/M $. By the maximality of $M$, it follows that $R_1M=X$. Therefore,
	$$P_1=P_1\cap (R_1M)=R_1(P_1\cap M)=R_1Q_1=R_1,$$ and hence $Q_1$ is a maximal normal subgroup of $P_1$. Note that $$Q_{i+1}=Q_1^{m^i}\leq M^{m^i}\leq M.$$ It follows that $L=Q_1\times\dots Q_5\lhd X$ is a subgroup of $M$. Therefore, $M/L$ is a proper normal subgroup of $X/L$ that is not contained in $N/L$. But $N/L$ is the unique minimal normal subgroup of $X/L\cong (P/Q)\wr \A_5$ as $P/Q$ is non-abelian simple, leading to a contradiction.
	Finally, by the uniqueness of $N$, we have that $N$ is a characteristic subgroup of $X$.
\end{proof}

For a group $G$ and a subgroup $H$ of $G$, the \emph{core} of $H$ in $G$, denoted ${\sf core}_G(H)$, is the intersection of all conjugates of $H$ in $G$, and is a normal subgroup of $G$. 
The quotient $G/{\sf core}_G(H)$ acts faithfully and transitively on the set of right cosets $[G:H]$, and is thus a subgroup of the symmetric group on $[G:H]$.
\begin{lemma}\label{Aut(Gamma)}
	$\Aut(\Ga)=X$.
\end{lemma}
\begin{proof}
Let $A=\Aut(\Ga)$. Denote the vertices $H$ and $Hg$ in $\Ga$ by $u$ and $ v$, respectively. Then the vertex-, edge-, arc- stabilizers are as follows: $X_u=\l h_1,h_2\r,$ $X_{\{u,v\}}=\l h_2, g\r$ and $X_{(u,v)}=\l h_2\r$.

We first claim that $N\lhd A$.
If  $X\lhd A$, then by Lemma~\ref{N:largestenormal}, $N$ is normal in $A$.
If $\mathsf{core}_A(X)<X$, by Lemma~\ref{N:largestenormal}, we have $\mathsf{core}_A(X) \leq N$, and
\[X/\mathsf{core}_A(X) = (N/\mathsf{core}_A(X)).(X/N).\]
Since $\Ga$ is a cubic graph, $|A_u|\leqslant 48$ by Tutte's Theorem~\cite{Tutte}, and hence $|A:X|=|A_u:X_u|\leq 8$. 
Then we have $A/\mathsf{core}_A(X)\leq \S_8$. 
Since $N$ is a perfect group, $N/\mathsf{core}_A(X)$ is also a perfect group. 
However, every subgroup of $\S_8$ has at most one non-solvable composition factor and $X/\mathsf{core}_A(X) $ already has one, namely the quotient $X/N\cong \A_5$. 
Hence $N/\mathsf{core}_A(X)=1$, and thus $N=\mathsf{core}_A(X)\lhd A$. 

By Lemma~\ref{lem:covers}, the quotient graph $\Ga_N$ is the Petersen graph. 
Let $K$ be the kernel of the action of $A$ on the vertex set of $\Ga_N$. 
Then $K\geq N$, and the set of $K$-orbits on the vertex set of $\Ga$  coincides with the set of $N$-orbits. 
This implies that $\Ga_N=\Ga_K$. 
By~\cite[Theorem~4.1]{Praeger1}, both $N$ and $K$ act semiregularly on the vertex set of $\Ga$, which implies that $K=N$.
Therefore, $A/N=A/K \leq \S_5\cong \Aut(\Ga_N)$. 
Suppose $|A:X|=2$. 
Then $A/N \cong \S_5$, and $\Ga$ is $3$-arc-transitive. 
Note that the arc stabilizer and edge stabilizer of the action of $\S_5\cong A/N$ on $\Ga_N$ are $\ZZ_2^2$ and $\D_8$, respectively. 
Then $A_{(u,v)} \cong \ZZ_2^2$ and $A_{\{u,v\}} \cong \D_8$, as $\Ga$ is a cover of $\Ga_N$.
Choose $a \in A_{(u,v)} \setminus X_{(u,v)}$. 
Since $A_{(u,v)}$ is abelian, $h_2^a = h_2$. 
As $A_{\{u,v\}}$ is non-abelian, $g^a \neq g$. 
Since $a$ normalizes $X_{\{u,v\}}$, we conclude that $g^a = gh_2$.
Therefore,
\[\C_N(g)^a = \C_{N^a}(g^a) = \C_N(gh_2).\]
Let $n=(p_1,p_2,p_3,p_4,p_5)\in N$, where $p_i\in P$ for $1\leq i\leq 5$.
Direct computations give
\[\begin{split}
		ng=&(p_1,p_2\mu,p_3,p_4,p_5 \mu)(14)(25)\mbox{ and }\\
		gn=&(p_4,\mu p_5, p_3, p_1,\mu p_2)(14)(25).\\
\end{split}\]
If $n\in \C_N(g)$, then $p_4=p_1$ and $p_5=p_2^\mu$. This gives
\begin{equation}\label{eqn:cng}
	\C_N(g)= \{(p_1, p_2,p_3,p_1,p_2^\mu)\in N \mid  p_1,p_2,p_3\in P\}
\end{equation}
Note that $gh_2=(1,\mu,\mu,\mu,1)(15)(24)$.
Similar computations for $n(gh_2)$ and $(gh_2)n$ yield that
\[\begin{split}
    n(gh_2)=&(p_1,p_2\mu,p_3\mu,p_4\mu,p_5)(15)(24)\mbox{ and }\\
    (gh_2)n=&(p_5,\mu p_4, \mu p_3, \mu p_2,p_1)(15)(24).\\
\end{split}\]
If $n\in \C_N(gh_2)$, then $p_5=p_1$, $p_4=p_2^\mu$ and $p_3$ centralizes $\mu$. This gives
\begin{equation}\label{eqn:cngh2}
	\C_N(gh_2)= \{(p_1, p_2,p_3,p_2^\mu,p_1)\in N \mid p_1,p_2\in P\mbox{ and } p_3\in \C_P(\mu)\}.
\end{equation}

Equations~\ref{eqn:cng} and \ref{eqn:cngh2} imply that $\C_N(g)\cong P^3$ and $\C_N(gh_2)\cong P^2\times \C_P(\mu)$.
Since $P=\l \lambda, \mu\r$ is non-abelian, $\C_P(\mu)$ is a proper subgroup of $P$. Consequently, $\C_N(g)\not\cong \C_N(gh_2)$, a contradiction. Hence, we have $A=X$.
\end{proof}

\begin{rem}
	By Lemma~\ref{Aut(Gamma)}, $\Ga$ is a $2$-arc-transitive but not $3$-arc-transitive cover of the Petersen graph. Using Magma~\cite{MAGMA}, one can find a subgroup $H\cong \S_3\times \S_2$ and an element $g$ in $\A_5\wr \S_5$ such that the coset graph $\Cos(\A_5\wr \S_5, H, HgH)$ is a $3$-arc-transitive minimal cover of the Petersen graph with transformation group $\A_5^5$. However, it does not seem easy to directly generalize this example to the case of an infinite family.
\end{rem}

\begin{lemma}\label{Non-Cayley}
	$\Ga$ is not a Cayley graph.
\end{lemma}
\begin{proof}
    Suppose, for a contradiction, that $\Ga$ is a Cayley graph. 
	By Lemma~\ref{Aut(Gamma)}, $\mathsf{Aut}(\Ga) = X$ contains a 
	vertex-regular subgroup $R$. Because $\Ga_N$ is the Petersen graph, 
	$|R| = 10|N|$ and $|X:R| = 6$, which implies that 
	$X/\mathsf{core}_X(R) \leq \S_6$. 
	By Lemma~\ref{N:largestenormal}, $\mathsf{core}_X(R) \leq N$.
	Furthermore, $X/\mathsf{core}_X(R) \cong (N/\mathsf{core}_X(R)) . (X/N)$, 
	and since $N$ is perfect, the quotient $N/\mathsf{core}_X(R)$ is 
	also perfect. Every subgroup of $\S_6$ contains at most one non-solvable 
	composition factor. The group $X/\mathsf{core}_X(R)$ already contains 
	one such factor, namely the quotient $X/N \cong \A_5$. 
	This forces the perfect group $N/\mathsf{core}_X(R)$ to be trivial, 
	yielding $N = \mathsf{core}_X(R)$.
	It follows that $N < R$. Thus, $R/N$ acts regularly on the vertex set of $\Ga_N$. This contradicts the fact that the Petersen graph has no vertex-regular subgroup. 
\end{proof}

\begin{lemma}\label{girth}
	The graph $\Ga$ has girth $10$.
\end{lemma}
\begin{proof}
	We denote the images of elements, subgroups, or subsets in $X$ under the natural homomorphism to $X/N$ by placing a bar over them.
	It is convenient to use the same symbol to denote the graph homomorphism from $\Ga$ to $\Ga_N$, which maps each vertex $Hx \in V\Ga$ to $\overline{Hx}=\overline{H}\overline{x}=\braket{(123),(12)(45)} \overline{x}$.
	
	Set $g_1=g, g_2=g^{h_1}$ and $g_3=g^{h_1^2}$. It is easily seen that $HgH$ is a disjoint union of $Hg_i$ for $1\leq i\leq 3$. Thus, for any vertex $Hx\in V\Ga$, the neighbourhood of $Hx$ is $\{Hg_ix| 1\leq i\leq 3\}$. It follows that $(v_0, v_1,\dots, v_s)$ is an $s$-arc of $\Ga$, if and only if there is an element $x\in G$ and a sequence $(y_1, y_2,\dots, y_s)$ where $y_i\in \{g_1, g_2, g_3\}$ such that
	\begin{align*}
		v_0&=Hx,& v_1&=Hy_1x, &&\dots &,v_i&=Hy_iy_{i-1}\dots y_1x,&& \dots  &,v_s&=Hy_sy_{s-1}\dots y_1x,
	\end{align*}
	and $y_i\neq y_{i+1}$ for $1\leq i \leq s-1$.
	Note that an $s$-arc $(v_0, v_1,\dots, v_s)$ is a cycle of length $s$ if and only if $v_s=v_0$ and $v_0, v_1,\dots, v_{s-1}$ are pairwise distinct.
	
	We first construct a $10$-cycle in $\Ga$, which implies that the girth of $\Ga$ is at most $10$.
	By calculation, we have $\braket{g_1,g_2}=\braket{g_1g_2}{:}\braket{g_1}\cong D_{10}$ and $\braket{g_1,g_2}\cap \braket{h_1,h_2}=1$.
	Now, set $y_{2i-1}=g_2, y_{2i}=g_1$ for $1\leq i\leq 5$, and let $v_0=H$, and $v_i=Hy_iy_{i-1}\dots y_1(1\leq i \leq 10)$ be vertices in $\Ga$. 
	Then $(v_0,v_1, \dots, v_{10})$ is a $10$-arc in $\Ga$ and 
	\[v_{10}=H(g_2g_1)^5=H((g_1g_2)^{-1})^5=H((g_1g_2)^{5})^{-1}=H=v_0.\]
	If there are two integers $0\le i<j\le 9$ such that $v_i=v_j$, then we have 
    \[Hy_jy_{j-1}\dots y_{i+1}=H.\]
	That is, $x=y_jy_{j-1}\dots y_{i+1}$ is an element in  $\braket{g_1,g_2} \cap \braket{h_1,h_2}=1$. 
	Hence $x=1$. 
	As $0\le i<j\le 9$ and $y_{i+1}, \dots , y_j\in \{g_1, g_2\}$, this is impossible. 
	Hence, $v_0, \dots, v_9$ are pairwise distinct, and $(v_0,v_1, \dots, v_{10})$ gives rise to a $10$-cycle in $\Ga$.

    Suppose that $\Ga$ has girth $s$, where $s < 10$.
    Since $\Ga$ is a cover of the Petersen graph $\Sig$, $s$ is at least $5$.
    Let $(v_0, v_1, \dots, v_s)$ be an $s$-cycle in $\Ga$. 
	By the $2$-arc-transitivity of $\Ga$, without loss of generality, we may assume $v_0=v_s=H$, $v_1=Hg_1$, $v_2=Hg_2g_1$. 
	Moreover, there is a subscript sequence $(i_1, i_2, \dots, i_s)$,
	\begin{equation}\label{eqns1}
		i_j\in \{1,2,3\}(1\le j\le s),~ i_1=1,~ i_2=2,~\mbox{and}~ i_j\ne i_{j-1}(3\le j\le s),
	\end{equation}
	such that $v_j=Hg_{i_j}g_{i_{j-1}}\dots g_{i_1}(1\le j\le s)$.
	As $H=v_0=v_s=Hy_sy_{s-1}\dots y_1$,  $g_{i_1}g_{i_2}g_{i_3}\dots g_{i_s}=(g_{i_s}g_{i_{s-1}}\dots g_{i_1})^{-1}\in H$.
	This gives
	\begin{equation}\label{eqns2}
		\overline{g_{i_1}}\cdot \overline{g_{i_2}}\cdot \overline{g_{i_3}} \dots \overline{g_{i_s}} =\overline{g_{i_1}g_{i_2}g_{i_3}\dots g_{i_s}}
		\in \overline H=\braket{(123),(12)(45)}.
	\end{equation}
	Using Magma~\cite{MAGMA}, it is easy to find all possible sequences $(i_1, i_2, \dots, i_s)$ satisfying conditions~(\ref{eqns1}) and (\ref{eqns2}) for $5\leq s\leq 9$.
	See the first column in Table~\ref{tab:list}. Further calculations allow us to uniquely rewrite the element $g_{i_1}g_{i_2}g_{i_3}\dots g_{i_s}$ in the form $nh$, where $n\in N$ and $h\in H$. 
    See the second and the third columns in Table~\ref{tab:list}. 
    As $g_{i_1}g_{i_2}g_{i_3}\dots g_{i_s}\in H$, we have $n=1$.
	Note that $P=\braket{\mu, \lambda}$ is a perfect group, and therefore the elements
	\[\mu, \lambda\mu\lambda^{-1}, [\lambda, \mu], [\mu, \lambda], [\mu, \lambda^{-1}]\]
	in $P$ are not the identity element.
	Also, since
	\[P=\braket{\lambda,\lambda^\mu}=\braket{\lambda,(\lambda^{-1})^\mu}=\braket{\lambda^{-1},(\lambda^{-1})^\mu},\]
	we have the commutator $[\lambda,\lambda^\mu]=\lambda^{-1}\mu\lambda^{-1}\mu\lambda\mu\lambda\mu$, which is not the identity in $P$. Similarly,
	\[ [\lambda,(\lambda^{-1})^\mu]=\lambda^{-1}\mu\lambda\mu\lambda\mu\lambda^{-1}\mu~\mbox{and}~[(\lambda^{-1})^\mu,\lambda^{-1}]=\mu\lambda\mu\lambda\mu\lambda^{-1}\mu\lambda^{-1}\]
	are also not the identity.
	It follows that the framed entries in Table~\ref{tab:list} are not the identity. 
	Thus, for each row in Table~\ref{tab:list}, we have $n\neq 1$, a contradiction.
	Hence, the shortest cycle in $\Ga$ is of length $10$, and $\Ga$ is of girth $10$.
\end{proof}
\begin{table}[ht]
	\scriptsize
	\begin{tabular}{c| c c }
		
		\hline
		$(i_1, i_2, \dots, i_s) $& $n$ & $h$\\ \hline
		$( 1, 2, 1, 2, 1 )$&  $(\mu^\lambda , \qd{\mu},\mu^{\lambda\mu},\mu^\lambda,\mu)$   & $ h_2h_1$    \\
		$( 1, 2, 3, 1, 2 )$&$  (\qd{[\lambda,\mu]}, [\mu,\lambda^{-1}], [\mu,\lambda^{-1}], [\mu,\lambda^{-1}], [\mu, \lambda] )  $&$  h_1^2   $\\ \hline
		$( 1, 2, 1, 3, 1, 2 )$&$  (\lambda^{-1}\mu\lambda^{-1}\mu\lambda^{-1}\mu,\mu\lambda\mu\lambda^{-1}\mu,\qd{\mu},\lambda\mu\lambda^{-1},\mu\lambda\mu\lambda\mu\lambda)  $&$    h_2 $\\
		$( 1, 2, 3, 2, 1, 2 )$&$   (\qd{\mu},\mu\lambda\mu\lambda\mu\lambda,\mu\lambda\mu\lambda^{-1}\mu,\mu,\mu\lambda^{-1}\mu\lambda^{-1}\mu\lambda^{-1}) $&$   h_1h_2  $\\ \hline
		$( 1, 2, 1, 3, 2, 1, 2 )$&$  (\lambda^{-1}\mu\lambda^{-1}\mu\lambda^{-1},\mu\lambda\mu\lambda\mu\lambda\mu,\qd{[\mu, \lambda]},\lambda\mu\lambda^{-1}\mu,\mu\lambda\mu\lambda^{-1})   $&$   h_1  $\\
		$( 1, 2, 3, 2, 3, 2, 1 )$&$  (\lambda\mu\lambda^{-1},\mu\lambda\mu\lambda^{-1}\mu,\qd{\mu},\mu\lambda^{-1}\mu\lambda\mu,\mu\lambda^{-1}\mu\lambda\mu)  $&$   h_2  $\\ \hline
		$( 1, 2, 1, 2, 3, 1, 3, 1 )$&$   (\lambda^{-1}\mu\lambda\mu\lambda^{-1}\mu\lambda\mu ,\qd{[\mu, \lambda]},\mu\lambda^{-1}\mu\lambda,\lambda^{-1}\mu\lambda^{-1}\mu\lambda^{-1},\mu\lambda\mu\lambda^{-1})   $&$   h_1  $\\
		$( 1, 2, 1, 2, 3, 2, 1, 3 )$&$  (\lambda^{-1}\mu\lambda\mu\lambda\mu\lambda^{-1},\mu\lambda^{-1}\mu\lambda\mu,\mu\lambda^{-1}\mu\lambda^{-1}\mu\lambda^{-1},\qd{\mu},\mu\lambda\mu\lambda^{-1}\mu)  $&$  h_1h_2   $\\
		$( 1, 2, 1, 3, 2, 3, 2, 1 )$&$   (\lambda^{-1}\mu\lambda\mu,\qd{[\mu,\lambda^{-1}]},\mu\lambda^{-1}\mu\lambda^{-1}\mu\lambda^{-1}\mu,\lambda\mu\lambda\mu\lambda,\mu\lambda\mu\lambda^{-1}\mu\lambda\mu\lambda^{-1}) $&$  h_1^2   $\\
		$( 1, 2, 3, 1, 3, 1, 2, 3 )$&$ (\lambda^{-1}\mu\lambda^{-1}\mu\lambda^{-1},\mu\lambda\mu\lambda^{-1}\mu\lambda\mu\lambda^{-1},\mu\lambda\mu\lambda\mu\lambda\mu,\mu\lambda^{-1}\mu\lambda,\qd{[\mu,\lambda^{-1}]})   $&$  h_1    $\\
		$( 1, 2, 3, 1, 3, 2, 3, 2 )$&$  (\lambda^{-1}\mu\lambda,\mu\lambda\mu\lambda^{-1}\mu\lambda^{-1}\mu\lambda\mu,\mu\lambda\mu\lambda\mu\lambda,\mu\lambda\mu\lambda^{-1}\mu,\qd{\mu})    $&$  h_1^2h_2   $\\
		$( 1, 2, 3, 2, 3, 1, 2, 1 )$&$  (\lambda\mu\lambda^{-1}\mu,\mu\lambda\mu\lambda\mu\lambda\mu,\mu\lambda^{-1}\mu\lambda,\qd{[\mu,\lambda]},\mu\lambda^{-1}\mu\lambda\mu\lambda^{-1}\mu\lambda)  $&$ h_1     $\\ \hline
		$( 1, 2, 1, 2, 3, 1, 2, 1, 3 )$&$( \qd{[\lambda,\mu]} ,\mu\lambda^{-1}\mu\lambda^{-1}\mu\lambda^{-1}\mu,\mu\lambda^{-1}\mu\lambda^{-1}\mu\lambda^{-1}\mu,\lambda^{-1}\mu\lambda\mu,\mu\lambda\mu\lambda\mu\lambda\mu)    $&$   h_1^2   $\\
		$( 1, 2, 1, 2, 3, 2, 3, 1, 3 )$&$ (\qd{\lambda^{-1}\mu\lambda\mu\lambda\mu\lambda^{-1}\mu},1,\mu\lambda^{-1}\mu\lambda^{-1}\mu\lambda\mu\lambda,1,\mu\lambda\mu\lambda^{-1}\mu\lambda^{-1}\mu\lambda)     $&$  1   $\\
		$( 1, 2, 1, 3, 1, 3, 1, 2, 3 )$&$  (\lambda^{-1}\mu\lambda^{-1}\mu\lambda\mu\lambda,\mu\lambda\mu\lambda^{-1}\mu\lambda\mu\lambda^{-1}\mu,\mu\lambda\mu\lambda^{-1}\mu,\qd{\mu},\mu\lambda\mu\lambda^{-1}\mu)  	  $&$     h_1h_2 $\\
		$( 1, 2, 1, 3, 1, 3, 2, 3, 2 )$&$  (\qd{\lambda^{-1}\mu\lambda^{-1}\mu\lambda\mu\lambda\mu},\mu\lambda\mu\lambda^{-1}\mu\lambda^{-1}\mu\lambda,1,1,\mu\lambda\mu\lambda\mu\lambda^{-1}\mu\lambda^{-1}) 	  $&$  1   $\\
		$( 1, 2, 1, 3, 2, 3, 1, 2, 1 )$&$  (\qd{\mu},\mu\lambda\mu\lambda\mu\lambda,\mu\lambda^{-1}\mu\lambda^{-1}\mu\lambda^{-1},\lambda\mu\lambda\mu\lambda^{-1}\mu\lambda^{-1},\mu\lambda\mu\lambda^{-1}\mu)    $&$  h_1h_2   $\\
		$( 1, 2, 3, 1, 3, 1, 3, 2, 3 )$&$   (\lambda^{-1}\mu\lambda^{-1}\mu\lambda\mu\lambda,\mu\lambda\mu\lambda^{-1}\mu\lambda\mu\lambda^{-1}\mu,\mu\lambda\mu\lambda^{-1}\mu,\qd{\mu},\mu\lambda\mu\lambda^{-1}\mu) $&$   h_1h_2   $\\
		$( 1, 2, 3, 1, 3, 2, 1, 2, 3 )$&$   (\lambda^{-1}\mu\lambda^{-1}\mu\lambda^{-1}\mu,\mu\lambda\mu\lambda^{-1}\mu,\mu\lambda\mu\lambda\mu\lambda^{-1}\mu\lambda^{-1}\mu,\mu\lambda\mu\lambda\mu\lambda,\qd{\lambda\mu\lambda^{-1}})   $&$  h_2   $\\
		$( 1, 2, 3, 2, 1, 3, 1, 2, 3 )$&$  (\mu\lambda\mu\lambda^{-1}\mu,\mu\lambda\mu\lambda\mu\lambda^{-1}\mu\lambda^{-1}\mu,\mu\lambda\mu\lambda\mu\lambda,\mu\lambda^{-1}\mu\lambda^{-1}\mu\lambda^{-1},\qd{\mu})    $&$    h_1^2h_2  $\\
		$( 1, 2, 3, 2, 1, 3, 2, 3, 2 )$&$   (\qd{[\mu,\lambda]},\mu\lambda\mu\lambda\mu\lambda\mu,\mu\lambda\mu\lambda\mu\lambda\mu,\mu\lambda^{-1}\mu\lambda,\mu\lambda^{-1}\mu\lambda^{-1}\mu\lambda^{-1}\mu)   $&$    h_1   $\\
		$( 1, 2, 3, 2, 3, 1, 3, 1, 2 )$&$   (1,\qd{\mu\lambda\mu\lambda\mu\lambda^{-1}\mu\lambda^{-1}},1,\mu\lambda^{-1}\mu\lambda^{-1}\mu\lambda\mu\lambda,\mu\lambda^{-1}\mu\lambda\mu\lambda\mu\lambda^{-1})     $&$  1   $\\ \hline
	\end{tabular}
	
	\normalsize
	\vspace*{0.3cm}
	\caption{The subscript sequences $(i_1, i_2, \dots, i_s) $ and corresponding group elements $n\in N$ and $h\in H$ with $g_{i_1}g_{i_2}g_{i_3}\dots g_{i_s}=nh$.}
	\label{tab:list}
\end{table}

\vskip0.2in
\thanks{{\bf Acknowledgements.}
	The authors are grateful to the anonymous referees for their valuable comments and suggestions that have helped to improve the paper.

	This work was supported by the National Natural Science Foundation of China (No. 11931005, 12101518, 11901512), Foundation for Cultivated Young Talents of Fujian Province (No.2026350493), the Fundamental Research Funds for the Central Universities (No. 20720210036, 20720240136, 20720260155),  and the Natural Science Foundation of Yunnan Province (No. 202401AT070275). 
	}

\end{document}